\documentclass[11pt, reqno, english]{amsart}  
\usepackage[utf8]{inputenc}
\usepackage[T1]{fontenc}
\usepackage{amsmath,amsthm}
\usepackage{amsfonts,amssymb}
\usepackage{url}
\usepackage{mathtools}  
\usepackage[colorlinks=true,urlcolor=blue,linkcolor=red,citecolor=magenta]{hyperref}
\usepackage{enumerate, paralist}
\usepackage{tikz-cd}
\usepackage[left=1in,right=1in,top=1in,bottom=1in]{geometry}

\usepackage{gensymb,graphicx,float,wasysym}

\theoremstyle{plain}
\newtheorem{theorem}{Theorem}[section]

\newtheorem{lemma}[theorem]{Lemma}
\newtheorem{corollary}[theorem]{Corollary}

\theoremstyle{definition}
\newtheorem{definition}[theorem]{Definition}

\newtheorem{remark}[theorem]{Remark}

\newcommand{\R}{\mathbb{R}}

\newcommand{\N}{\mathbb{N}}
\newcommand{\D}{\mathcal{D}}

\newcommand{\be}{\begin{equation}}
\newcommand{\ee}{\end{equation}}

\newcommand{\eqn}[1]{\begin{align*}#1\end{align*}}

\newcommand{\tand}{\text{ and }}

\newcommand{\stat}{\text{SN}}

\newcommand{\odd}{\text{ odd }}
\newcommand{\even}{\text{ even }}
\newcommand{\ohalf}{\frac{1}{2}}

\newcommand{\snake}{\textit{Snake}}
\newcommand{\povn}{\frac{\pi}{n}}

\begin{document}

\title{What can you draw?}



\author{Florian Frick}
\address[FF]{Dept.\ Math.\ Sciences, Carnegie Mellon University, Pittsburgh, PA 15213, USA}
\email{frick@cmu.edu} 

\author{Fei Peng}
\address[FP]{Dept.\ Math.\ Sciences, Carnegie Mellon University, Pittsburgh, PA 15213, USA}
\email{ief@cmu.edu}


\begin{abstract}
\small
We address the problem of which planar sets can be drawn with a pencil and eraser. The pencil draws any union of black open unit disks in the plane~$\R^2$. The eraser produces any union of white open unit disks. You may switch tools as many times as desired. Our main result is that drawability cannot be characterized by local obstructions: A bounded set can be locally drawable, while not being drawable. We also show that if drawable sets are defined using closed unit disks the cardinality of the collection of drawable sets is strictly larger compared with the definition involving open unit disks.
\end{abstract}

\date{April 2, 2020}
\maketitle

\section{Introduction}

The second author raised the following deceptively simple question: What can you draw? Your canvas is the plane~$\R^2$---colored white to begin with---and you are given two tools to draw with: a pencil (or brush), which produces a black unit disk wherever it meets the canvas, and an eraser, which produces a white unit disk. There are no further restrictions on your artistic freedom: You may raise the tool off the canvas, that is, there is no continuity requirement for the centers of disks you draw, and you can switch tools as many times as desired. 

More precisely, for a set $A \subset \R^2$ denote its open $1$-neighborhood, the union of all open unit disks with center in~$A$, by 
$$N(A) = \{x \in \R^2 \: : \: |x-a| < 1 \ \text{for some} \ a \in A\}.$$
A subset of the plane that can be drawn without the use of the eraser is of the form~$N(A_1) = D_1$ for some $A_1 \subset \R^2$. We can now ``erase'' the set $N(A_2)$ for some $A_2 \subset \R^2$ to obtain $D_2 = N(A_1) \setminus N(A_2)$. Using the pencil a second time we can draw any set of the form $D_3 = (N(A_1) \setminus N(A_2)) \cup N(A_3)$, from which we can erase $N(A_4)$ to produce~$D_4$, and so on. 
We say that we produced the set $D_k$ after $k$ steps.
Denote by $\D_1$ the sets we can draw in one step: The collection of sets $N(A_1)$ for $A_1\subseteq \R^2$. Similarly, $\D_2=\{N(A_1)\setminus N(A_2): A_1, A_2 \subset \R^2\}$. In general, $$\D_n=\begin{cases}\{D\cup N(A_n): D\in \D_{n-1}, A_n\subseteq \R^2\}&(n\text{ is odd})\\\{D\setminus N(A_n): D\in \D_{n-1}, A_n\subseteq \R^2\}&(n\text{ is even})\end{cases}.$$

We are interested in the collection of drawable sets $\D = \bigcup_{n = 1}^\infty \D_n$. We will refer to any set in $\D$ as \emph{drawable}. For $A\in \D$, its presence in $\D$ will be \emph{witnessed} by some~$A_1,\dots,A_n$ in the above form, namely: $$A=((((N(A_1) \setminus N(A_2))\cup N(A_3))\setminus N(A_4))\dots$$

\begin{figure}[H]
    \centering
    \includegraphics[scale=0.46]{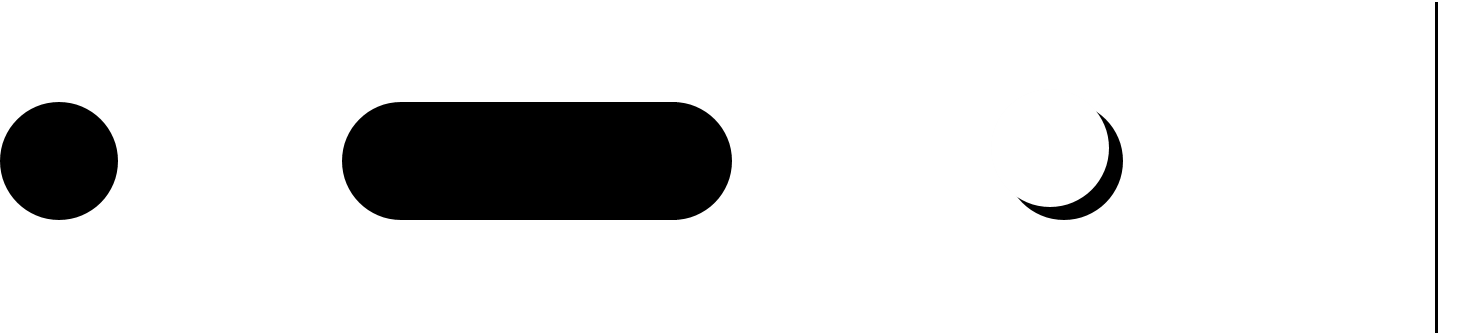}
    \caption{Four simple examples of drawable sets.}
    \label{fig:easyg}
\end{figure}

The choice that our drawing tools produce open unit disks (instead of closed unit disks) is arbitrary and we will investigate a second model of drawable sets, where open $1$-neighborhoods are replaced by their non-strict counterparts
$$N_{\le}(A) = \{x \in \R^2 \: : \: |x-a| \le 1 \ \text{for some} \ a \in A\}.$$
We avoid the terminology closed $1$-neighborhood since $N_{\le}(A)$ is not necessarily a closed set, for example if $A$ is an open unit disk. Replacing each $N(A_j)$ by $N_{\le}(A_j)$ in the definition of $\D$ we get the collection of \emph{closed-disk drawable} sets~$\D_{\le}$.

We can make some observations about drawable sets, such as every closed convex set is drawable and any convex set is closed-disk drawable; see Section~\ref{sec:convex} for the simple proofs. The purpose of the present manuscript is to derive a more surprising phenomenon, namely that being a drawable set is not a local condition. First, we mention that local obstructions to drawability exist:

\begin{theorem}
\label{thm:chessboard}
    A $2 \times 2$ chessboard, that is, the set $[-1,0] \times [-1,0] \cup [0,1] \times [0,1]$, is neither drawable nor closed-disk drawable.
\end{theorem}

Call a set $B \subset \R^2$ \emph{locally drawable} if every point $x \in \R^2$ has a neighborhood $U$ such that there is a drawable set $D\in \D$ such that $U \cap B$ is equal to $U \cap D$. That is, if we zoom in close to any point in~$D$, the part of the set we see is indistinguishable from a drawable set. Clearly, any drawable set is locally drawable. 

The left image of Figure~\ref{fig:moreimposg} shows a simple example of a set that is locally drawable, but not drawable: Round off the corners of a $2\times 2$ chessboard to separate the two black squares of the chessboard, thus making it locally drawable. If this smoothing is sufficiently sharp, that is, we round off with a curve of curvature strictly larger than one, any unit disk touching the curve from the inside of the black region will extend past the curve. We thus need to use the eraser to achieve this curvature, but the eraser will interfere with the other black region. So neither black region can be drawn last. This is a quick outline of a proof that such a chessboard with rounded corners is not drawable. We find this unsatisfactory, as it feels that we won on a technicality: First, we made the boundary of our drawing so sharp that the pencil does not fit into it; second, the obstruction is still somewhat local, that is, the two black regions need to be close enough that erasing around one region interferes with the other.

Here we rectify both of these shortcomings. We construct an example of a simple closed curve in the plane with curvature less than one everywhere (so that pencil and eraser can locally approximate it from either side), such that the region bounded by it is not drawable; see Theorem~\ref{thm:curv1}. And we identify truly global obstructions to drawability; for given $r>0$ we construct obstructions to drawability that are found in an annulus of inradius~$r$ (and depend on the annulus closing up). We need additional language for a precise statement, which we thus postpone to Theorem~\ref{thm:totdisundraw}. The general obstruction we exhibit to prove Theorem~\ref{thm:totdisundraw} is the same used to prove the following:

\begin{theorem}\label{thm:curv1}
    There is a Jordan loop $\gamma$ in the plane with curvature strictly between $-1$ and~$1$, such that the interior region $R$ of $\gamma$ is neither drawable nor closed-disk drawable. However, $R$ is locally drawable and locally closed-disk drawable.
\end{theorem}

\begin{figure}[H]
     \centering
     \begin{tikzpicture}
         \clip (-1.2,-1.22) rectangle (15.3,1.2);
         \fill [black] (0,0) rectangle (1cm,1cm);
         \fill [black] (-1cm,-1cm) rectangle (0,0);
         \node[text width=5cm,align=center] at (3.3cm,0)  {Can be open or closed\\$\notin$  $\D$, $\notin$ $\D_{\le}$\\Not locally drawable};
         \node[inner sep=0pt] (whitehead) at (9cm, 0)
             {\includegraphics[height=2.5cm]{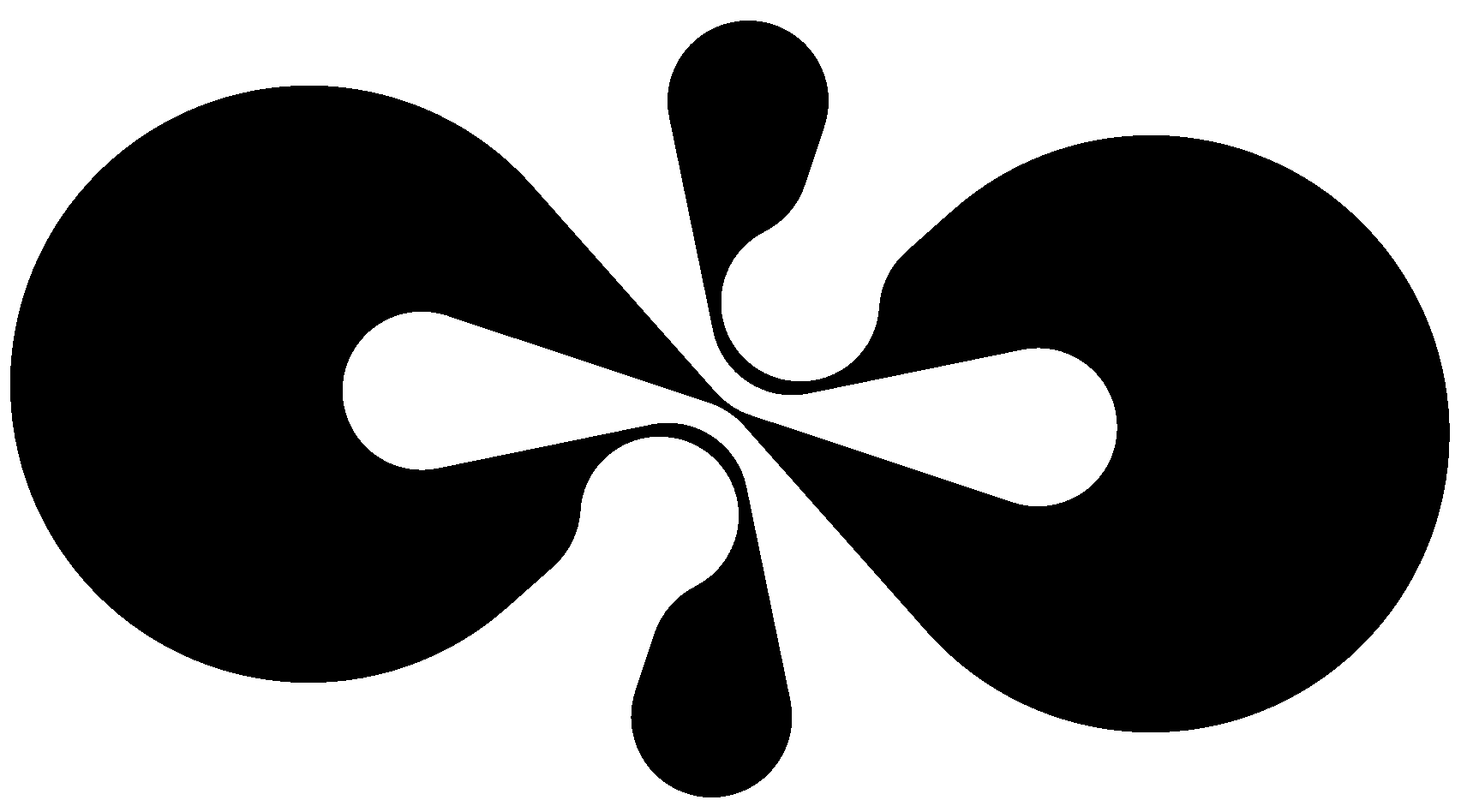}};
         \node[text width=5cm,align=center] at (13.3cm,0)  {Can be open or closed\\Bounded; curvature $<$ 1\\$\notin$ $\D$, $\notin$  $\D_{\le}$\\Is locally drawable};
     \end{tikzpicture}
     \caption{An undrawable $2 \times 2$ chessboard and a Jordan curve of curvature $< 1$ that bounds a locally drawable, yet undrawable region, the ``snake.''}
     \label{fig:imposg}
\end{figure}
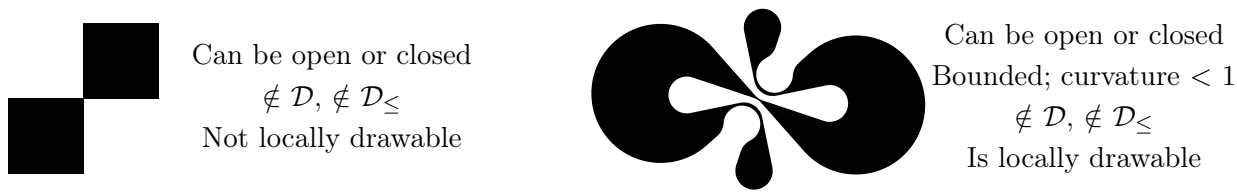

A set bounded by a Jordan loop with curvature strictly between $-1$ and $1$ is locally drawable (and locally closed-disk drawable); see Theorem~\ref{thm:curvature_bound}. This is because we may approximate the curve from either side with disk of radius at least one, and thus pencil and eraser ``fit into'' the curve. This is Blaschke's rolling ball theorem~{\cite[p.~114]{blaschke1916}} that a unit disk fits into any convex curve with curvature at most one.

Any drawable set is a Borel set, that is, in the $\sigma$-algebra generated by open sets in the plane, 
and Theorem~\ref{thm:chessboard} shows the existence of Borel sets that are not drawable. Here we show:

\begin{theorem}
\label{thm:closed-disk}
    Any closed-disk drawable set is a Lebesgue set. Not every Lebesgue subset of $\R^2$ is closed-disk drawable, but $\D_{\le}$ has the same cardinality as the set of Lebesgue subsets of~$\R^2$. In particular, $|\D_{\le}| > |\D|$.
\end{theorem}

The first part is an immediate consequence of~\cite{balcerzak1999uncountable}. While the two models of what it means to be a drawable set are very similar---using open unit disks versus closed unit disks---the model where drawing tools leave a closed unit disk produces a larger cardinality of drawable sets.

To the authors' knowledge the notion of drawability has not been investigated earlier. There is, however, the related concept of Dynkin system: A non-empty family of subsets of a set $X$ is called \emph{Dynkin system} if it is closed under taking complements and countable disjoint unions. Keleti~\cite{keleti1999} showed that the Dynkin system generated by open balls of radius at least one in~$\R^d$, $d\ge 3$, does not contain all Borel sets. Keleti and Preiss~\cite{keleti2000} showed that the Dynkin system generated by all open balls in a separable infinite-dimensional Hilbert space does not contain all Borel sets. Finally, Zelen\'y~\cite{Zeleny} showed that the Dynkin system generated by balls in $\R^d$ contains all Borel sets.

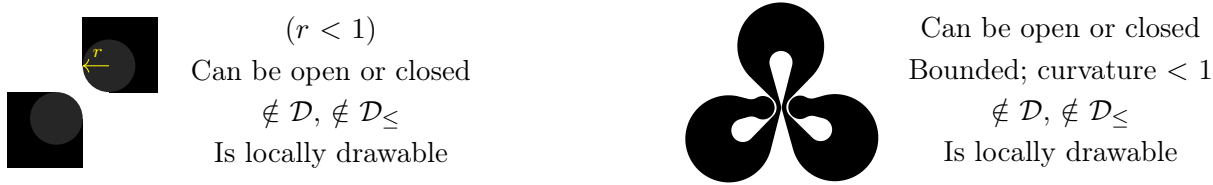
\begin{figure}[H]
     \centering
     \begin{tikzpicture}
         \clip (-1.2,-1.22) rectangle (15.3,1.2);
         \fill [black] (0,0) rectangle (1cm,1cm);
         \fill [black] (-1cm,-1cm) rectangle (0,0);
         \fill [white] (-0.35cm,-0.35cm) rectangle (0.35cm,0.35cm);
         \fill [black!70!gray] (-0.35cm,-0.35cm) circle (0.35cm);
         \fill [black!70!gray] (0.35cm,0.35cm) circle (0.35cm);
         \draw [yellow,->] (0.35cm,0.35cm) -- (0,0.35cm);
         \node[yellow] at (0.2cm, 0.52cm) {\tiny $r$};
         \node[text width=5cm,align=center] at (3.3cm,0)  {($r$ $<$ $1$)\\Can be open or closed\\$\notin$  $\D$, $\notin$ $\D_{\le}$\\Is locally  drawable};
         \node[inner sep=0pt] (whitehead) at (9.3cm, 0)
             {\includegraphics[height=2.4cm]{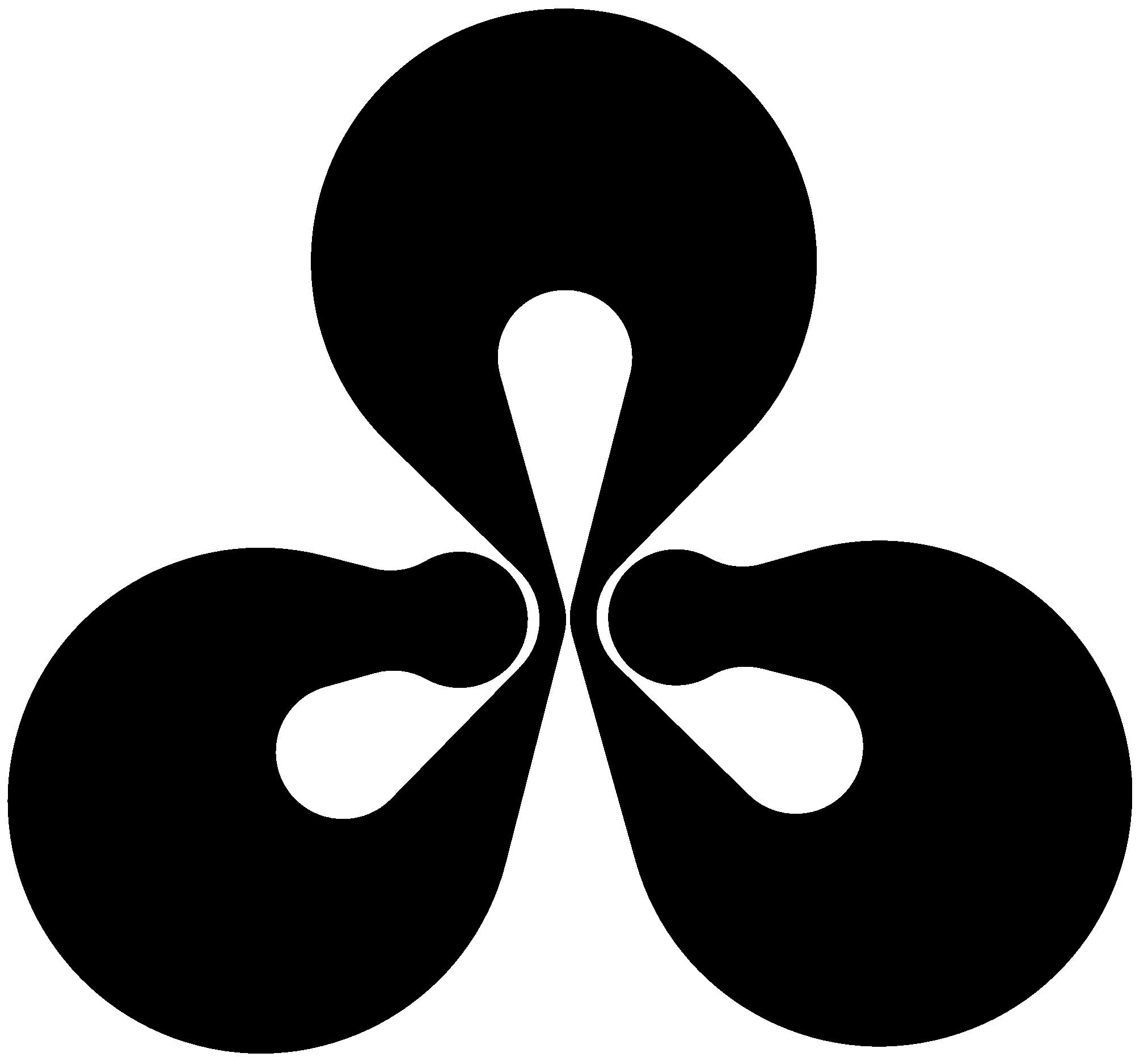}};
         \node[text width=5cm,align=center] at (13cm,0)  {Can be open or closed\\Bounded; curvature $<$ 1\\$\notin$ $\D$, $\notin$  $\D_{\le}$\\Is locally drawable};
     \end{tikzpicture}
     \caption{Some locally drawable but undrawable sets. Their non-drawability follows from the general obstruction given in Theorem~\ref{thm:totdisundraw}. The construction of the second set---``octopus''---is similar to the snake in Figure~\ref{fig:imposg}.}
     \label{fig:moreimposg}
\end{figure}

\section{Properties of drawable sets}\label{sec:convex}

In this section we collect some simple properties of drawable and closed-disk drawable sets, and prove Theorem~\ref{thm:closed-disk}. Recall that a set $A\subset \R^2$ is \emph{convex} if for any two $x,y \in A$ the line segment connecting $x$ and $y$ is entirely within~$A$. For two vectors $x, y \in \R^2$ we denote their inner product $x_1y_1+x_2y_2$ by~$\langle x,y \rangle$.

\begin{theorem}
\label{thm:closed-convex}
    Any closed convex set in $\R^2$ is drawable.
\end{theorem}

\begin{proof}
    Any open halfspace, that is, any set of the form $H = \{x \in \R^2 \: : \: \langle x,y \rangle > a\}$ for some $y \in \R^2$ of norm~$1$ and $a \in \R$, is a union of open unit disks. Namely, $H$ is the set~$N(A)$, where $A$ is the set of~$z+\lambda y$ with $\langle z,y \rangle = a$ and $\lambda \ge 1$, that is, $A$ is the set of points in $H$ at distance at least one to the line $\{x \in \R^2 \: : \: \langle x,y \rangle = a\}$. In a first step we can color the plane black. In a second step we can erase any union of open halfspaces. This means that any intersection of closed halfspaces is drawable. This is precisely the collection of closed convex sets.
\end{proof}

The condition that the convex set be closed in order to be drawable is indeed needed. In fact, most convex sets are not drawable.
We will show this now. 

A set $A \subset \R^2$ is a \emph{Borel set} if it is contained in the $\sigma$-algebra generated by open sets in~$\R^2$. Recall that a non-empty set system is called $\sigma$-algebra if it is closed under taking complements and under taking countable unions.

\begin{theorem}
\label{thm:Borel}
    Every drawable set is a Borel set.
\end{theorem}

\begin{proof}
    Any set of the form $N(A)$ for $A\subset \R^2$ is open as a union of open disks, and thus every set in $\mathcal D_1$ is a Borel set. The claim that every element of $\D = \bigcup_{n = 1}^\infty \D_n$ is Borel as well now follows by a simple induction, since sets in $\mathcal D_n$ are obtained from sets in $\mathcal D_{n-1}$ either by taking complements with open sets or by taking the union with an open set.
\end{proof}

\begin{corollary}
\label{cor:cardinality-drawable}
    The cardinality of the collection of drawable sets $|\mathcal D|$ is strictly less than the cardinality of the collection of convex sets in the plane. In particular, most convex sets are not drawable.
\end{corollary}

\begin{proof}
    There are at most as many drawable sets as there are Borel sets by Theorem~\ref{thm:Borel}. The cardinality of the set of Borel sets is~$2^{\aleph_0}$, the cardinality of real numbers; see~{\cite[Thm.~3.3.18]{Srivastava_1998}}.
    However, the set of convex sets in the plane has the same size as the power set of the reals, which is strictly larger than~$2^{\aleph_0}$. To see this observe that any set that fits between the open unit disk centered at the origin and the closed unit disk centered at the origin is convex. That is, let $U$ be any subset of the unit circle~$S^1$. Then $\{x \in \R^2 \: : \: |x| < 1\} \cup U$ is convex. There are as many subsets of~$S^1$ as subsets of the reals.
\end{proof}

\begin{theorem}
    Any convex set in $\R^2$ is closed-disk drawable.
\end{theorem}

\begin{proof}
    We begin by showing that any closed convex set is closed-disk drawable. The proof is essentially the same as for Theorem~\ref{thm:closed-convex}. With the difference that now, given some $H = \{x \in \R^2 \: : \: \langle x,y \rangle > a\}$ for $y \in \R^2$ of norm~$1$ and $a \in \R$, we have to represent it as $N_{\ge}(A)$ for some~$A \subset \R^2$, that is, as a union of closed unit disks. The set $H$ is simply the union of closed unit disks centered at~$z+\lambda y$ with $\langle z,y \rangle = a$ and $\lambda > 1$, that is, $A$ is the set of points in $H$ at distance strictly greater than one from the line $\{x \in \R^2 \: : \: \langle x,y \rangle = a\}$.
    
    Now given some convex set $C \subset \R^2$, first realize its closure $\overline C$ as a closed-disk drawable set. We then have to delete certain boundary points of~$\overline C$, namely all points in $\overline C \setminus C$. The points in $\overline C \setminus C$ are contained in the union of closed unit disks that stay entirely within the complement of~$C$. Indeed, for any point $x_0 \in \overline C \setminus C$ consider a supporting line~$\ell$, that is, a line that is disjoint from the interior of~$C$ and contains~$x_0$. If $\ell$ is defined by the equation $\langle x,y \rangle = a$ for $y \in \R^2$ of norm one and $a\in \R$, then the closed unit disk centered at $x_0 + y$ contains $x_0$ and is entirely contained within the complement of~$C$.
\end{proof}

A set $A \subset \R^2$ that differs from a Borel set in a subset of a set of Lebesgue measure zero is called \emph{Lebesgue set}. The collection of Lebesgue sets form a $\sigma$-algebra, since countable unions of measure-zero sets have measure zero.

\begin{proof}[Proof of Theorem~\ref{thm:closed-disk}]
    Any (not necessarily countable) union of closed unit disks is a Lebesgue set~\cite[Thm.~1.1]{balcerzak1999uncountable}. Since Lebesgue sets form a $\sigma$-algebra, this implies that any closed-disk drawable set is a Lebesgue set. The cardinality of the set of Lebesgue sets is the same as the cardinality of the power set of~$\R$, which is equal to the cardinality of convex sets in $\R^2$ by the proof of Corollary~\ref{cor:cardinality-drawable}. All of these sets are closed-disk drawable, showing that there are as many closed-disk drawable sets as Lebesgue sets. Since each drawable set is a Borel set by Theorem~\ref{thm:Borel} and the set of Borel sets has the cardinality~$2^{\aleph_0}$ of the reals, we have that $|\D| < |\D_{\ge}|$. 
    
    It remains to exhibit an example of a Lebesgue subset of~$\R^2$ that is not closed-disk drawable. Observe that for any closed-disk drawable set $A \in \D_{\ge}$ there is a closed unit disk in~$A$ or a closed unit disk in the complement of~$A$. This is because every set is finalized in finitely many steps and the last step either drew a black unit disk in~$A$ or erased a white unit disk. A sufficiently fine checkerboard pattern is an example of a subset $A$ of $\R^2$ such that neither $A$ nor its complement contain a (closed) unit disk. (For a less trivial, bounded example of a Lebesgue set that is not closed-disk drawable---namely a $2 \times 2$ chessboard already suffices---see Theorem~\ref{thm:chessboard}, proven in the next section.)
\end{proof}

\section{Non-drawability of the $2 \times 2$ chessboard}

For a drawable set $A\in\D$ witnessed by sets $A_1, \dots, A_n$, and any point $x \in \R^2$ there is a last time where the color of~$x$ in the process of drawing $A$ changed. We call this the stationary number~$\stat(x)$ of~$x$. We give the precise definition here.

\begin{definition}
For $A\in \D$ witnessed by $A_1,\dots,A_n \subset \R^2$ and a point $x\in \R^2$ define
$$\stat(x):=\begin{cases}\min\{k\text{ odd}:x\in N(A_k)\tand\forall\even k'>k,x\notin N(A_{k'})\}&(x\in A)\\\min\{k\text{ even} :x\in N(A_k)\tand\forall\odd k'>k,x\notin N(A_{k'})\}&(x\notin A)\end{cases}$$
\end{definition}

\begin{definition}\label{defi:encircle}
A collection of points $S\subseteq \R^2$ is said to \textit{encircle} $T\subseteq \R^2$ if it is impossible for an open unit disk to touch any point in $T$ without touching any point in $S$; i.e., $$\forall x\in \R^2, (B(x,1)\cap S=\emptyset\rightarrow B(x,1)\cap T=\emptyset).$$
\end{definition}

We note that if $S_1$ encircles $T_1$ and $S_2$ encircles $T_2$, then $S_1\cup S_2$ encircles $T_1\cup T_2$.
If $S$ encircles $T$ and we would like to change the color of any point in~$T$, we must also change the color of some point in~$S$. So with this notion in hand, we can build local obstructions to a set being drawable. The following lemma is almost immediate.

\begin{lemma}\label{lem:maxsn}
If $S$ encircles $T$ and every pair $\{x,y\}\in S\times T$ has opposite colors, then $$\max_{x\in S}\stat(x)> \max_{y\in T}\stat(y).$$
\end{lemma}
\begin{proof}
Let $k=\max_{y\in T}\stat(y).$ Some unit disk centered in $A_k$ must touch some point $y$ in~$T$, that is, there is a $y\in T\cap N(A_k)$, so it must also cover some point $x$ in~$S$, that is, there is an $x\in S\cap N(A_k)$. At this step the color of $x$ and the color of $y$ are the same, but $x$ and $y$ have opposite colors at the end. Thus $x$ is not finalized at the $k$-th step, so $\stat(x)>k$.
\end{proof}

We need one more elementary geometric fact before we can prove Theorem~\ref{thm:chessboard}, that a $2 \times 2$ chessboard is not drawable. 

\begin{lemma}\label{lem:trape}
If an isosceles trapezoid has base lengths $a,b$ ($a<b$), leg length $c$ and height $h$ (that means $c^2=(b-a)^2/4+h^2$), then its circumradius is $$R = \frac{c\sqrt{ab+c^2}}{2h}.$$
\end{lemma}
\begin{proof}
Say trapezoid $ABCD$ is as described, with $|AB|=a$. Denote the circumcenter as $O$ and the midpoint of $BC$ as~$M$. Find $H$ on $CD$ such that $BH\perp CD$. Consider the line parallel to $AB$ through~$O$, and let $OL$ be the projection of $OM$ to this line. Note that $|OL|=(a+b)/4$ and $|OL|/|OM|=|BH|/|BC|=h/c$ and hence $|OM|=(a+b)c/4h$. As $O$ is the circumcenter, $OM\perp BC$; thus, \eqn{R&\ =\ \sqrt{|OM|^2+|BM|^2}\nonumber\\
&\ =\ \sqrt{\left(\frac{(a+b)c}{4h}\right)^2+\left(\frac{c}{2}\right)^2}\nonumber\\
&\ =\ \frac{c}{2h}\sqrt{\frac{(b+a)^2}{4}+h^2}\nonumber\\
&\ =\ \frac{c}{2h}\sqrt{ab+c^2}.\qedhere}
\end{proof}

\begin{proof}[Proof of Theorem~\ref{thm:chessboard}]
Suppose the set $[-1,0] \times [-1,0] \cup [0,1] \times [0,1]$ was drawable. We will construct an infinite sequence of points with strictly decreasing stationary number. This will be a contradiction.

Consider $B_{1a}:=(r\cos\theta,r\sin\theta)$, for some small $r<1$ and $\theta<\pi/4$. Let $B_{1b}=\begin{bsmallmatrix}0&1\\1&0\end{bsmallmatrix}B_{1a}$ (the reflection of $B_{1a}$ about $y=x$), $B_{1c}=-B_{1a},B_{1d}=-B_{1b}.$ Obtain the analogous points $W_{1a},W_{1b},W_{1c},W_{1d}$ by using $-\theta$ in place of $\theta$. Let $B_{2x}=\ohalf B_{1x},W_{2x}=\ohalf W_{1x}$ ($x\in\{a,b,c,d\}$). As shown in Figure \ref{fig:2x2pt}, the $B$ points are black and $W$ points are white.

\begin{figure}
\centering
    \includegraphics[scale=0.5]{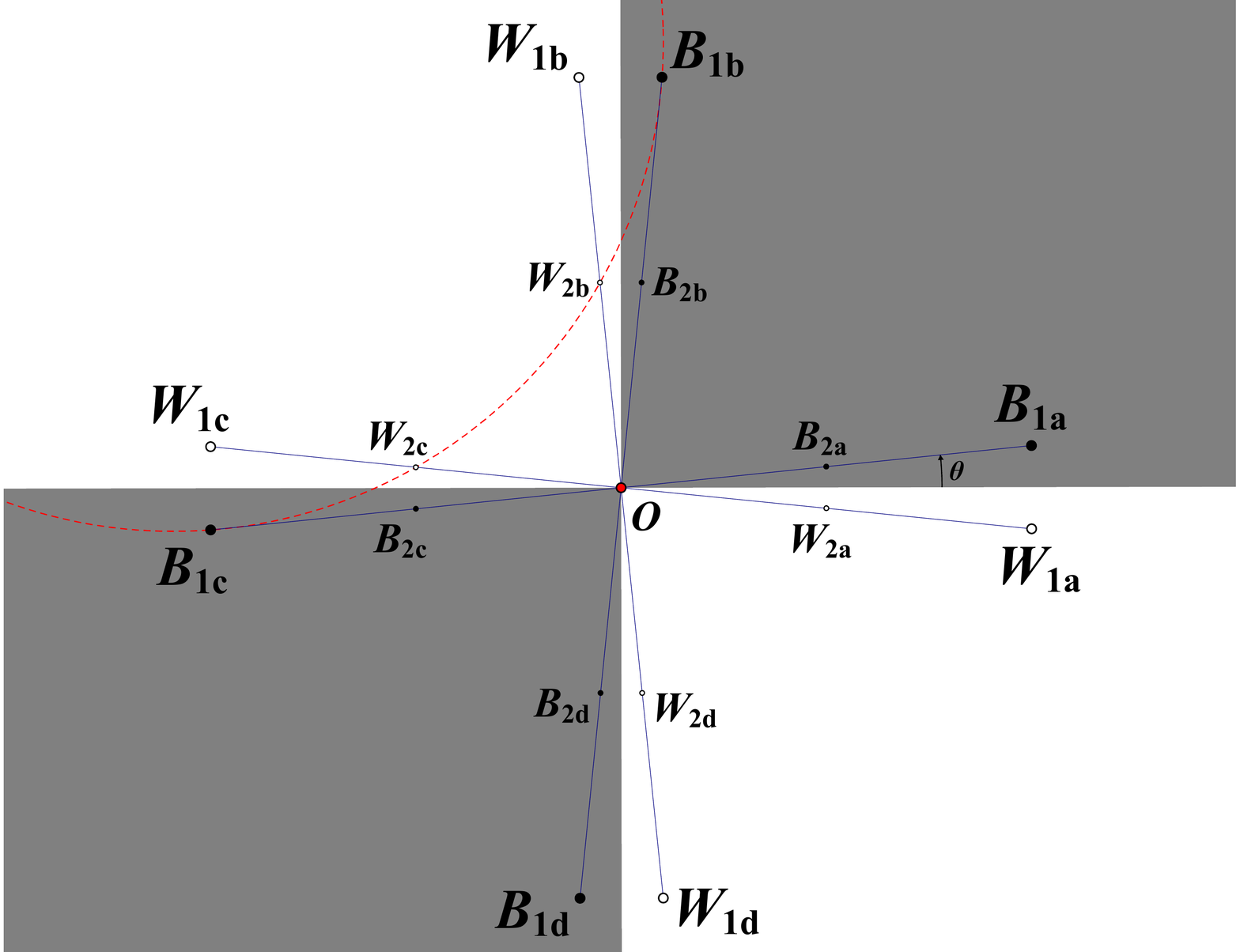}
    \caption{Points in the $2 \times 2$ chessboard.}
    \label{fig:2x2pt}
\end{figure}

We note that with sufficiently small $r$ and $\theta$, $\{B_{1a},B_{1b},B_{1c},B_{1d}\}$ encircles $\{W_{2a},W_{2b},W_{2c},W_{2d}\}$. The red dashed circle in Figure \ref{fig:2x2pt} demonstrates the largest disk that could touch the inner white points without touching the outer black ones. By Lemma \ref{lem:trape}, when $\theta\to 0$, its radius goes to $\sqrt{10}r/4$, which could be arbitrarily small with small $r$. (We just need it to be $<1$.) By symmetry, $\{W_{1a},W_{1b},W_{1c},W_{1d}\}$ encircles $\{B_{2a},B_{2b},B_{2c},B_{2d}\}$. Hence by Lemma \ref{lem:maxsn}, if $[-1,0]^2\cup[0,1]^2$ were drawable, then
\eqn{&\max(\stat(B_{1a}),\stat(B_{1b}),\stat(B_{1c}),\stat(B_{1d}), \stat(W_{1a}),\stat(W_{1b}),\stat(W_{1c}),\stat(W_{1d})) \\>\ &\max(\stat(B_{2a}),\stat(B_{2b}),\stat(B_{2c}),\stat(B_{2d}), \stat(W_{2a}),\stat(W_{2b}),\stat(W_{2c}),\stat(W_{2d})).}

Repeat this argument for $B_{3x} =\frac12 B_{2x}$ and  $W_{3x} = \frac12 W_{2x}$. The black points $B_{2x}$ of the second stage again encircle the white points $W_{3x}$ of the third stage, and likewise the $W_{2x}$ encircle the~$B_{3x}$. Thus the maximal stationary number of the points $B_{3x}$ and $W_{3x}$ is strictly smaller than that of the $B_{2x}$ and~$W_{2x}$. Repeat this process to get an infinite descending sequence of stationary numbers. Such an infinite descending chain of positive integers is not possible. This is a contradiction.
\end{proof}

\section{Undrawable sets with small curvature and global obstructions to drawability}

Here we show that if a region is bounded by a curve of small curvature, then it is locally drawable. We construct the {\snake}, a region whose boundary has small curvature, but that is not drawable. The obstruction to drawability we exhibit can be phrased in general terms, and this obstruction is ``global'' instead of ``local;'' see Theorem~\ref{thm:totdisundraw}.

First we recall some basic notions of the differential geometry of planar curves. We refer to do Carmo's book~\cite{doCarmo} for details. Let $\gamma$ be a simple smooth closed curve in the plane, parametrized by arc length, that is, $|\gamma'(s)|=1$ for all~$s$. Let $x_0 = \gamma(s_0)$ be a point on the trace of~$\gamma$. The curve $\gamma$ has a well-defined tangent line at~$x_0$. 
Rotate that tangent line by $90\degree$ in positive (i.e., counter-clockwise) direction to obtain the unit normal $n(s)$ of~$\gamma(s)$. 
Then since $\gamma'(s)$ is a unit vector, its derivative $\gamma''(s)$ is orthogonal to the tangent $\gamma'(s)$ for every~$s$. Thus $\gamma''(s) = k(s)n(s)$ for some function $k(s)$, called the (signed) curvature of~$\gamma$. The (unsigned) curvature is $\kappa(s) = |k(s)|$.

The following lemma may be seen as a special case of Blaschke's classical rolling ball theorem~{\cite[p.~114]{blaschke1916}}, which states that if two smooth regular (positively oriented) convex curves $\gamma_1$ and $\gamma_2$ touch in one point~$x$, where they have the same tangent vector, and the curvature of $\gamma_1$ is larger or equal to the curvature of~$\gamma_2$, then $\gamma_1$ is contained entirely within the region bounded by~$\gamma_2$. Moreover, if the curvature of $\gamma_1$ is strictly less than the curvature of~$\gamma_2$, then outside of the point~$x$, the curve $\gamma_1$ is contained in the interior of the region bounded by~$\gamma_2$.

We make no assumption on the convexity of curves, but locally every smooth regular curve is convex. Blaschke's theorem shows that we may choose $\varepsilon = 1$ in the lemma below. Since we do not need a sharp estimate on~$\varepsilon$, the lemma follows easily by Taylor expansion. We include the simple argument for the reader's convenience.

\begin{lemma}
\label{lem:rolling_disk}
    Let $\gamma\colon I \to \R^2$ be a smooth curve parametrized by arc length, defined on some compact interval~$I$, and let $s_0 \in I$. Assume $\kappa(s) < 1$ for all $s \in I$. Then there are two circles $C_1$ and $C_2$ of radius one with centers $\gamma(s_0) \pm n(s_0)$, which touch $\gamma$ at $\gamma(s_0)$ but do not contain $\gamma(s)$ for $s \in (s_0-\varepsilon, s_0+\varepsilon)$ for some $\varepsilon > 0$. Moreover, this $\varepsilon$ can be chosen independent of~$s_0$.
\end{lemma}

\begin{proof}   
    To simplify the notation we translate $I$ so that $s_0 = 0$. By applying an appropriate rigid motion we may assume that $\gamma(0) = (0,1)$ and $\gamma'(0) = \pm(1,0)$. By perhaps reversing orientation we may additionally assume that $\gamma'(0) =(1,0)$ and thus $n(0) = (0,1)$. There is a $\delta > 0$ such that the trace of $\gamma$ restricted to $s \in (-\delta, \delta)$ is the graph of a smooth function, say, $(s, f(s))$ is on the trace of $\gamma$ for $s \in (-\delta, \delta)$. We note that by the inverse function theorem $\gamma|_{[-\delta,\delta]}$ is the graph of a smooth function as long as the derivative of the first coordinate $\gamma_x'$ is non-zero everywhere. Since $\gamma_x'(0) = 1$ and $|\gamma_x''(s)| \le |\gamma''(s)| = \kappa(s) < 1$, we may choose $\delta > 0$ independently of~$s_0$. We chose the coordinate system in such a way that $f(0) = 1$ and $f'(0) = 0$. 
    
    The signed curvature of $\gamma$ at $(s,f(s))$ is 
    $$k(s) = \frac{f''(s)}{(1+f'(s)^2)^{3/2}}.$$
    Thus $f''(s) = k(s)\cdot (1+f'(s)^2)^{3/2}$, which is approximately $k(s)$ for small~$s$. By Taylor expanding $f$ we see that $f(s) = f(0) + f'(0)s +\frac12f''(\xi)s^2 = 1+\frac12f''(\xi)s^2$ for some $\xi$ between $0$ and~$s$.

    The relevant part of the circle of radius one with center $\gamma(0) - n(0) = (0,0)$ is the trace of the curve $C_1(s) = (s, \sqrt{1-s^2})$. Similarly, for the circle of radius one with center $\gamma(0) + n(0) = (0,2)$ we consider the curve $C_2(s) = (s, 2-\sqrt{1-s^2})$. We need to show that for small~$s$ we have $C_1(s) \le f(s) \le C_2(s)$. Equivalently, for small $s$ we need to show that
    $$\sqrt{1-s^2} \le 1+ \frac{f''(\xi)s^2}{2} \le 2-\sqrt{1-s^2}.$$
    This holds with equality for $s=0$, so we may assume $s\ne0$ from now on. These two inequalities can equivalently be expressed as
    $$\sqrt{1-s^2} \le 1 \pm \frac{f''(\xi)s^2}{2}.$$
    Squaring this and collecting all terms on the right we have to show that
    $0 \le (1\pm f''(\xi))s^2 +\frac14f''(\xi)^2s^4$. This is equivalent 
    to $0 \le 1\pm f''(\xi) +\frac14f''(\xi)^2s^2$, which is evidently true for $s$ close to~$0$ since $|k(s)| = \kappa(s) < 1$ and $f''(s) = k(s)\cdot (1+f'(s)^2)^{3/2}\approx k(s)$. Moreover, $0 \le 1\pm f''(\xi) +\frac14f''(\xi)^2s^2$ is a strict inequality for small but non-zero~$s$. We note that since the maximum unsigned curvature in the curve is less than 1 (by compactness), the threshold can be chosen independent of~$s_0$.
\end{proof}

In some sense, our notion of drawability may be seen as a sequential version of Blaschke's rolling ball theorem. We can now show that regions bounded by curves of small curvature are locally drawable.

\begin{theorem}
\label{thm:curvature_bound}
    Let $\gamma\colon I \to \R^2$ be a simple, smooth, closed curve, that is, $\gamma$ is a smooth embedding of a circle into the plane. Assume $|k(s)| < 1$ for all $s\in I$. Then the closed region bounded by $\gamma$ is locally drawable and locally closed-disk drawable.
\end{theorem}

\begin{proof}   
    Denote the closed region bounded by $\gamma$ by~$R$. Suppose $\gamma$ is positively oriented, so that $\gamma(s) + \lambda n(s)$ is in $R$ for all $\lambda \in [0, \delta)$ for some sufficiently small~$\delta > 0$ and every~$s$. Around any point $x$ in the interior of $R$ the set $R$ is easily seen to be locally drawable; after all, there is a small open set containing $x$ that is entirely contained in~$R$. By the same reasoning $R$ is locally drawable around any $x \notin R$. 
    
    For $x$ on the boundary of~$R$, say $x = \gamma(s_0)$, choose $\varepsilon > 0$ according to Lemma~\ref{lem:rolling_disk} (and independent of~$s_0$). By perhaps decreasing $\varepsilon$ such that $\varepsilon < \delta$, the $\varepsilon$-disk around $x$ intersects $\gamma$ only in points of the form $\gamma(s)$ for $s\in (s_0 - \varepsilon, s_0 + \varepsilon)$. Now the unit disks centered at $\gamma(s)+n(s)$ for $s\in (s_0 - \varepsilon, s_0 + \varepsilon)$ witness the local closed-disk drawability of $R$ around~$x$ by Lemma~\ref{lem:rolling_disk}. To see the local (open-disk) drawability, we erase the unit disks centered at $\gamma(s)-n(s)$ for $s\in (s_0 - \varepsilon, s_0 + \varepsilon)$.
\end{proof}

We now construct the ``snake'' in Figure~\ref{fig:imposg}, that is enclosed by a Jordan curve of curvature~$<1$, but is undrawable. We will choose $\kappa_0=1/1.001<1$ as the maximum curvature in the boundary, that is, the smallest osculating circle will have radius $r:=1.001$. We start by constructing a kite~$ABDC$, symmetric about the line segment~$AD$, such that $\angle ABC=15\degree,\angle CBD=30\degree$ (thus $\angle ABD =45\degree$) and the line segment $BD$ has length~$2r$. Next construct three circles with radius~$r$, centered at $B$, $D$ and $C$, respectively. The circle centered at $B$ is tangent to the circle centered at~$D$ in point~$M$. Similarly, the circle centered at $C$ is tangent to the circle centered at $D$ in point~$N$. Denote the intersection of $AB$ and the circle centered at $B$ by~$E$, and denote the intersection of $AC$ and the circle centered at $C$ by~$F$. Let $a_1$ be the shorter arc from $E$ to $M$ along the circle centered at~$B$, $a_2$ the longer arc from $M$ to $N$ along the circle centered at~$D$, and $a_3$ the shorter arc from $N$ to~$F$ along the circle centered at~$C$.

Construct the point $O$ such that $OE\perp AB, OF\perp AC$. Then $\angle EOF=30\degree.$ Extend $OE$ and $OF$ as rays $\ell_1,\ell_2$, and construct $\ell_3,\dots,\ell_{12}$ (all starting at~$O$) so that they together divide the space evenly into twelve parts in clockwise order.
Let $a_4$ be the minor $r$-arc (i.e., the circular arc with radius~$r$) that is tangent to $\ell_{12}$ and $\ell_3$, and let $a_5$ be the minor $r$-arc tangent to $\ell_{11}$ and $\ell_4$. Let $a_6,a_7$ be the major $r$-arcs tangent to $\ell_{12}\tand \ell_1$, and $\ell_3\tand \ell_4$, respectively. 

Let $E'$ be the tangent point of $a_6$ closer to~$A$. Let $F',S\tand T$ be on $\ell_2,\ell_{11}\tand \ell_5$ such that $|OE'|=|OF'|=|OS|=|OT|$. Let $a_8$ be the major arc that is tangent to $\ell_2$ at $F'$ and tangent to~$\ell_5$.

\begin{figure}[H]
    \centering
    \includegraphics[scale=0.45]{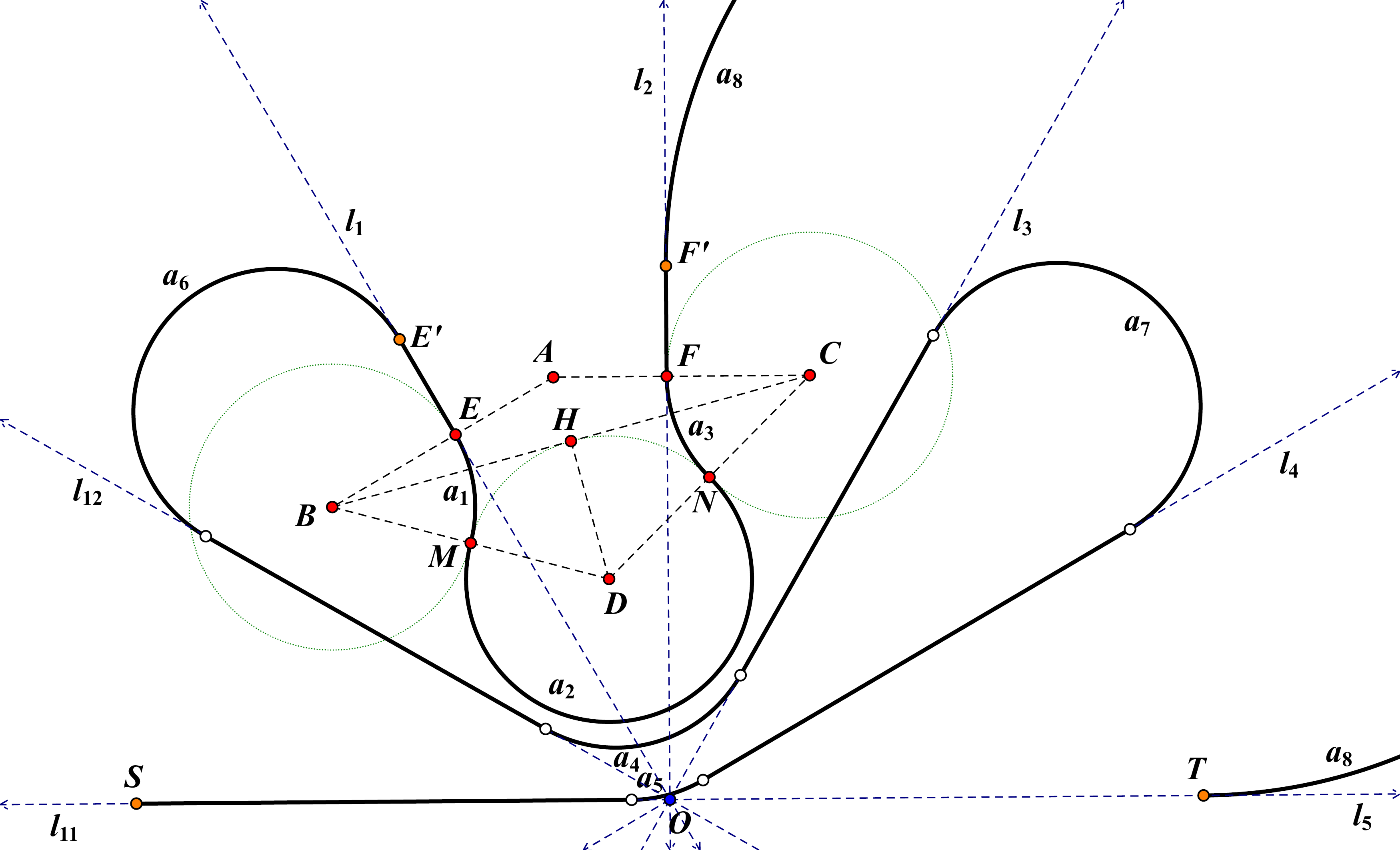}
    \caption{Construction of {\snake}.}
    \label{fig:snakeconstruc}
\end{figure}

As in Figure~\ref{fig:snakeconstruc}, a curved path connects $S, a_5, a_7, a_4, a_6, a_1, a_2, a_3, a_8 \tand T$ by line segments. 
Rotate this path by $180\degree$ around O, and they together form a simple, closed curve. The {\snake} is defined as the region enclosed by this curve. 

\begin{remark}\label{rek:c2prob}
The boundary of {\snake} is smooth everywhere except at the junctions between arcs and line segments (or arcs and arcs). At points where the curve is smooth, its curvature is always at most~$\kappa_0$. From the construction of {\snake} it is easily seen that it satisfies Lemma~\ref{lem:rolling_disk}, that is, for every point $x$ on the curve there are two unit disks to either side of the curve that intersect it in~$x$, and are otherwise disjoint from the curve in a neighborhood of~$x$. Thus {\snake} is locally drawable (and locally closed-disk drawable) with the same proof as in Theorem~\ref{thm:curvature_bound}.

If we wanted a smooth version of {\snake} with the same curvature bounds, we have to fuse the individual segments more carefully, continuously changing the curvature from zero along straight line segments to $\kappa_0$ along circle segments. This is precisely the ``track transition problem'' (or ``spiral easement'') encountered by railroad and highway engineers. See Dubins~\cite{dubins1957} for details. 
\end{remark}

Next, we show that {\snake} is neither drawable nor closed-disk drawable.

\begin{definition}
Let $A\subset \R^2$ be a set, $\ell \subset \R^2$ a ray emanating from the point~$P$, and $P_1$ and $P_2$ two points on~$\ell$ at distance $a$ and $b$ from~$P$, respectively. Suppose that $P_1$ is closer to $P$ than $P_2$, that is, $a < b$. Let $d > 0$. Consider the two rectangles (to either side of~$\ell$) with base $P_1P_2$, where the other side length is~$d$. If the interior of one of these rectangles is contained in~$A$, while the interior of the other rectangle does not intersect~$A$, we say that $A$ is \emph{dissected by $\ell$ at the interval $(a,b)$ with thickness~$d$}. If the rectangle contained in $A$ is in clockwise direction from~$P_1$, we say that the orientation of the dissection is \emph{clockwise}, and otherwise it is \emph{counterclockwise}.
A clockwise dissection is illustrated in Figure~\ref{fig:dissect}. 
\end{definition}
\begin{figure}
    \centering
    \includegraphics[scale=0.45]{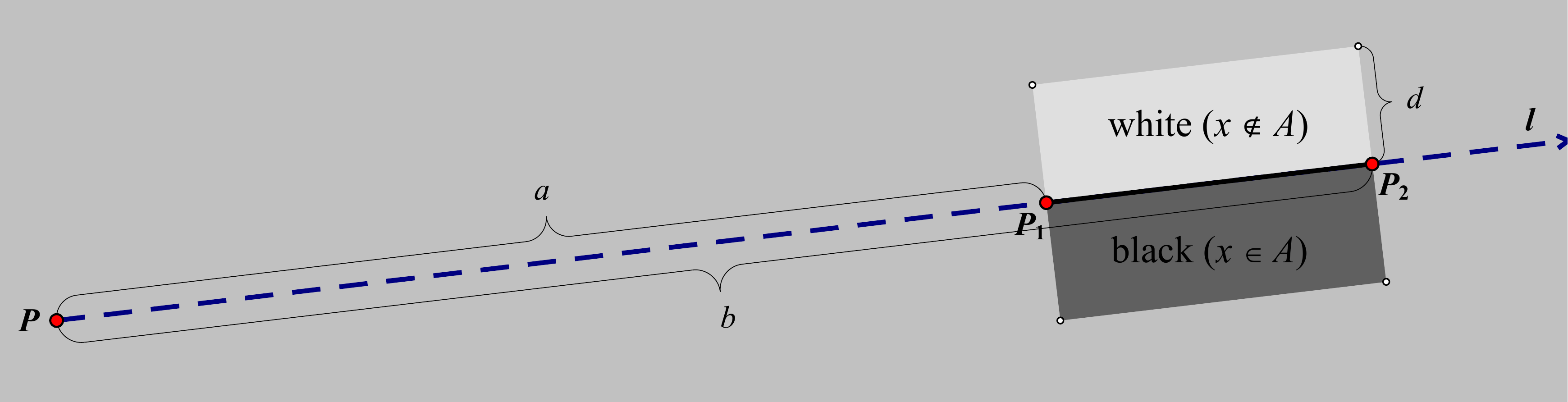}
    \caption{Here, $A$ is dissected by $\ell$ at $(a,b)$ with thickness $d$ and clockwise orientation.}
    \label{fig:dissect}
\end{figure}

\begin{definition}
For a positive even integer~$n$, a set $A\subset \R^2$ is \emph{totally $n$-dissected at interval $(a,b)$ with thickness $d$} if there are rays $\ell_1,\ell_2,\dots,\ell_n$ emanating from the same point in the given order that divide the plane evenly (i.e., into equal angles), such that $A$ is dissected by $\ell_i$ at $(a,b)$ with thickness~$d$ for every~$i$, and the adjacent rays have opposite orientations. An example is illustrated in Figure~\ref{fig:totaldissect}.
\end{definition}
\begin{figure}
    \centering
    \includegraphics[scale=0.3]{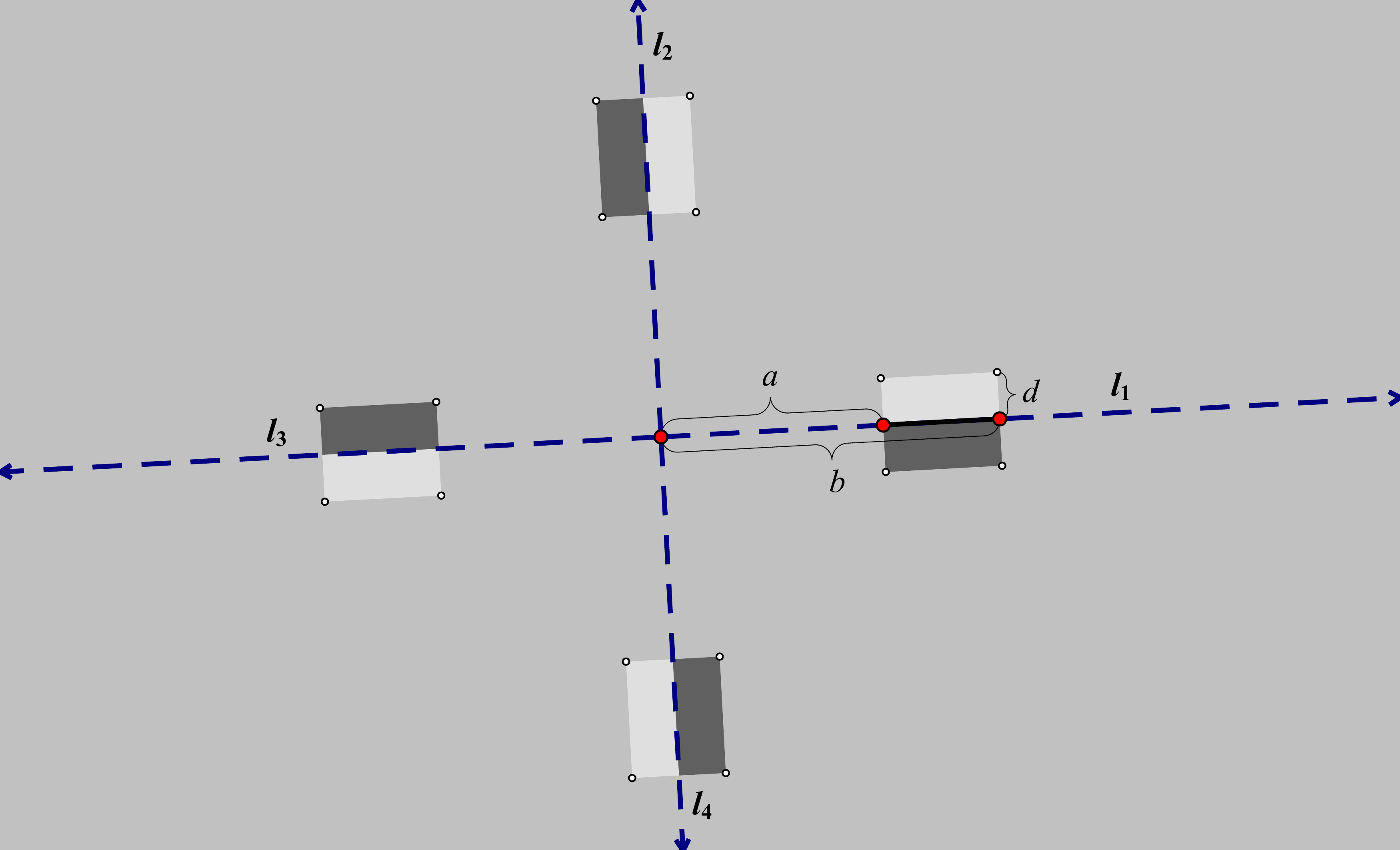}
    \caption{Here, $A$ is totally 4-dissected at $(a,b)$ with thickness $d$.}
    \label{fig:totaldissect}
\end{figure}

\begin{theorem}
\label{thm:totdisundraw}
If $A$ is totally $n$-dissected at $(a,b)$ with thickness~$d$, then if $a<\cot(\pi/n)$, the set $A$ is not drawable.
\end{theorem}

The bound in Theorem~\ref{thm:totdisundraw} is sharp: In a totally $n$-dissecting construction, a unit circle tangent to $\ell_1,\ell_2$ (i.e., a brush trying to squeeze in) will have its tangent points exactly $\cot(\pi/n)$ from $O$. Slide it along $\ell_1,\ell_2$ and replicate this process in each black zone, and we will draw a set that is totally $n$-dissected at $(\cot(\pi/n),\infty)$ with thickness two. This is illustrated in Figure~\ref{fig:illsharp}.

We first explain the proof idea of Theorem~\ref{thm:totdisundraw}. If $A$ is totally $n$-dissected at $(a,b)$ then there is a special point $O$ in~$A$, such that rotating $A$ by $\frac{4\pi}{n}$ around~$O$ leaves the set of dissected rectangles invariant, while rotating by $\frac{2\pi}{n}$ flips the colors in the set of dissected rectangles. We remark that this set is at distance $a$ from~$O$, and in particular, for large $n$, can be chosen arbitrarily far away from~$O$.

The condition that $a < \cot(\pi/n)$ now guarantees that cyclically consecutive dissected rectangles are sufficiently close that they must be drawn in order. Going around all $n$ rays, we then derive a contradiction once we arrive back at the first ray. Checking the relevant details is somewhat tedious but easy. We do this in the following.

\begin{figure}
    \centering
    \includegraphics[scale=0.3]{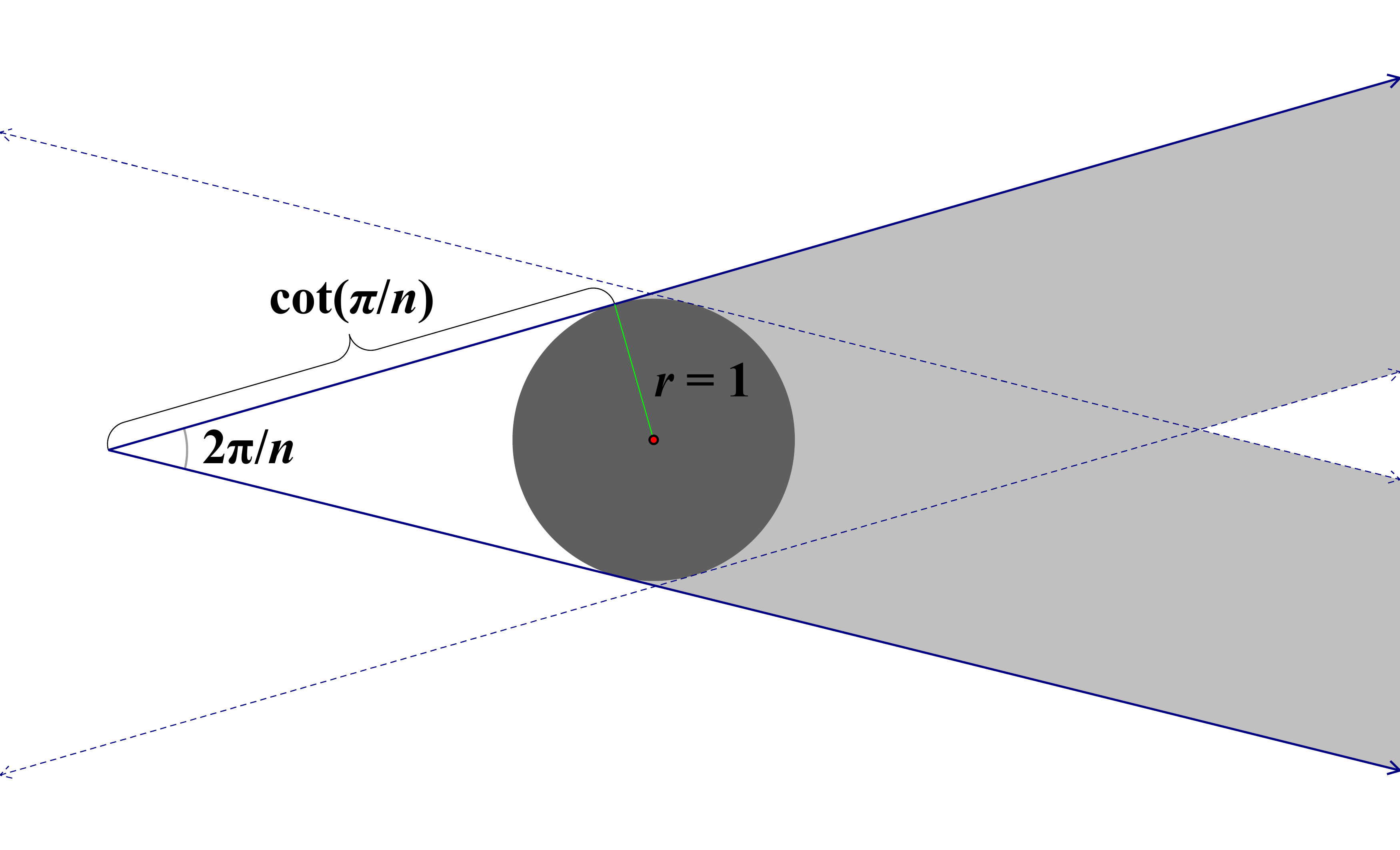}
    \caption{When $a\ge\cot(\pi/n)$, the set could be drawable.}
    \label{fig:illsharp}
\end{figure}

\begin{proof}[Proof of Theorem~\ref{thm:totdisundraw}]

\begin{figure}
    \centering
    \includegraphics[scale=0.45]{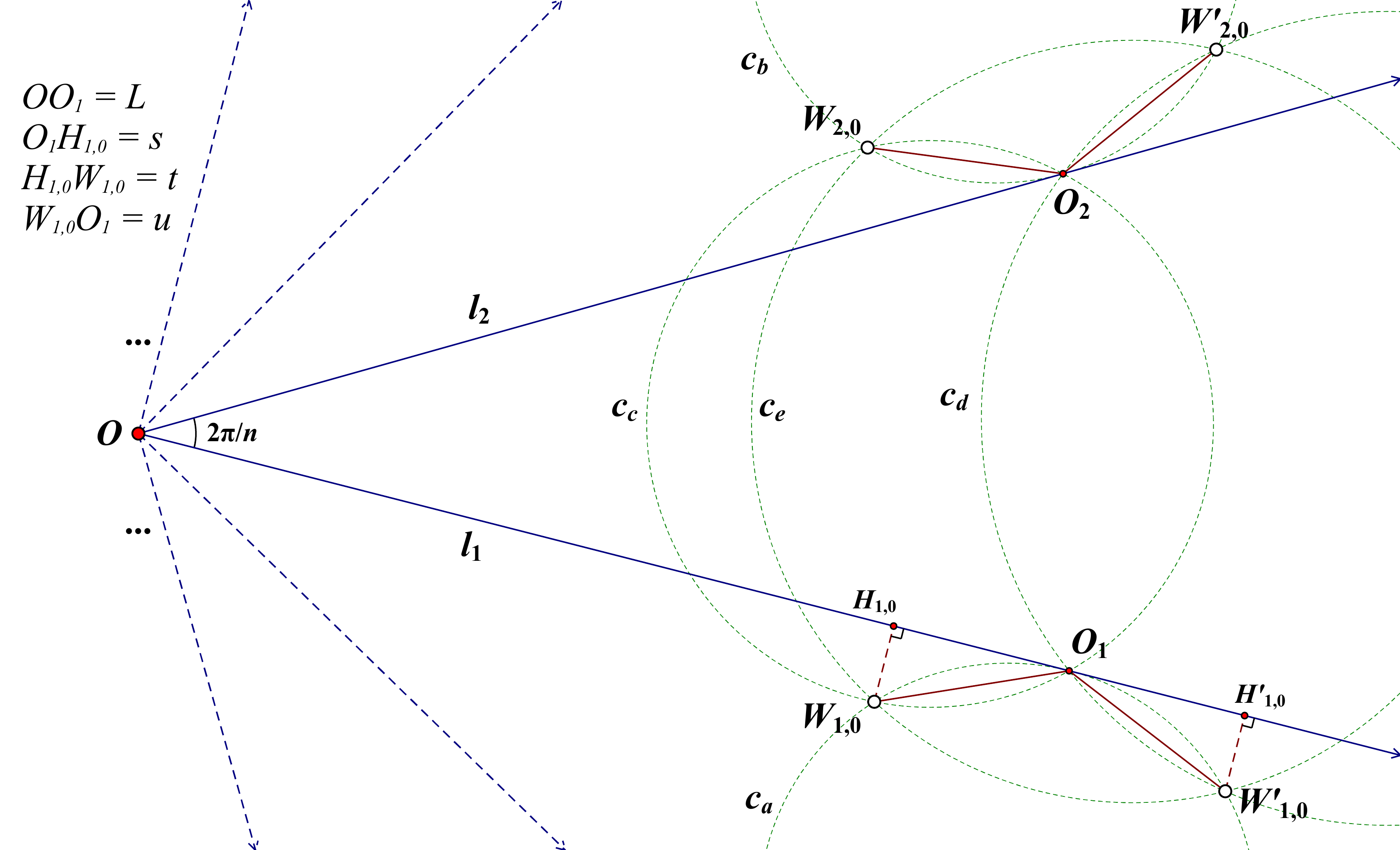}
    \caption{Some crucial points.}
    \label{fig:crucialpts}
\end{figure}

By definition, $A$ is dissected by $n$ rays emanating from the same point~$O$ that divide the plane evenly. Among those rays, find one which dissects $A$ with counterclockwise orientation, and denote it~$\ell_1$. Label the remaining rays $\ell_2,\dots,\ell_n$ in counterclockwise order.

We claim that there exists, for every $i\in\N$, white points $W_{1,i},W'_{1,i},W_{2,i},W'_{2,i} \in \R^2 \setminus A$ and black points $B_{1,i},B'_{1,i},B_{2,i},B'_{2,i} \in A$ such that for every $i\in\N$, \eqn{\{W_{1,i},W'_{1,i},W_{2,i},W'_{2,i}\}\text{ encircles }\{B_{1,i+1},B'_{1,i+1},B_{2,i+1},B'_{2,i+1}\}.}

Let $O_1$ and $O_2$ be on $\ell_1$ and~$\ell_2$, respectively, such that $|OO_1|=|OO_2|=L$, leaving $L$ to be specified later. Let $W_{1,0},W'_{1,0}$ be the two points that are (again, $s$ and $t$ to be specified later)
\begin{itemize}
    \item at distance $t$ from~$\ell_1$,
    \item closer to $\ell_n$ than~$\ell_2$, and 
    \item whose perpendiculars on $\ell_1$ are at distance $s$ from~$O_1$.
\end{itemize}  
Let $W_{1,0}$ be the point closer to~$O$. Let $W_{2,0},W'_{2,0}$ be the reflection points of $W_{1,0},W'_{1,0}$ about the bisector of $\ell_1$ and~$\ell_2$.

Our goal is to find $L,s,t$ such that points sufficiently close to $O_1$ or $O_2$ are encircled by $\mathcal W = \{W_{1,0},W'_{1,0},W_{2,0},W'_{2,0}\}$. Note that this is the case if the circles \eqn{
c_a&\ :=\ \text{the circle passing through }W_{1,0},O_1\tand W'_{1,0}\\
c_b&\ :=\ \text{the circle passing through }W_{2,0},O_2\tand W'_{2,0}\\
c_c&\ :=\ \text{the circle passing through }W_{1,0},O_1,O_2\tand W_{2,0}\\
c_d&\ :=\ \text{the circle passing through }W'_{1,0},O_1,O_2\tand W'_{2,0}\\
c_e&\ :=\ \text{the circle passing through }W_{1,0},W'_{1,0},W_{2,0}\tand W'_{2,0}
} all have radii smaller than~$1$. This is because any largest open disk that contains one of the $O_i$ but does not intersect $\mathcal W$ (which exists because $O_1,O_2$ are in the interior of the convex hull of $\mathcal W$) contains at least two points of $\mathcal W$ in its boundary---otherwise we may increase the radius of the disk. The five circles are the boundaries of the largest closed disks that contain at least two points of $\mathcal W$ in its boundary and at least one of the~$O_i$. Those circles are illustrated in Figure~\ref{fig:crucialpts}.

Let $u=|O_1W_{1,0}|=\sqrt{s^2+t^2}$. The radii of all five circles can be calculated by Lemma~\ref{lem:trape} ($c_a$ and~$c_b$, in particular, can be viewed as the circumcircle of an isosceles trapezoid with one base length~$0$):

{\small\eqn{
R(c_a)\ =\ R(c_b)&\ =\ \frac{u\sqrt{0\cdot 2s + u^2}}{2t}\ =\ \frac{u^2}{2t};\\
R(c_c)\ <\ R(c_d)&\ =\ \frac{u\sqrt{\left(L\sin\left(\povn\right)\right)\left(L\sin\left(\povn\right)+u\sin\left(\povn+\arctan\left(\frac{t}{s}\right)\right)\right)+u^2}}{u\cos\left(\povn+\arctan\left(\frac{t}{s}\right)\right)};\\
R(c_e)&\ =\ \frac{2s\sqrt{\left(L\sin\left(\povn\right)+u\sin\left(\povn-\arctan\left(\frac{t}{s}\right)\right)\right)\left(L\sin\left(\povn\right)+u\sin\left(\povn+\arctan\left(\frac{t}{s}\right)\right)\right)+(2s)^2}}{2s\cos\left(\povn\right)}.
}}

Now let $s^2\ll t\ll s\ll 1$. (For example, set $t=s^{1.5}$ and let $s\to 0^+$.) Then,

\eqn{
R(c_a)\ =\ R(c_b)&\ =\ \frac{s^2+t^2}{2t}\to 0;\\
R(c_c)\ <\ R(c_d)&\ \to\ \frac{\sqrt{\left(L\sin\left(\povn\right)\right)\left(L\sin\left(\povn\right)\right)}}{\cos\left(\povn+\arctan\left(\frac{t}{s}\right)\right)}\ \to\ L\cdot\tan\left(\povn\right);\\
R(c_e)&\ \to\ \frac{L\sin\left(\povn\right)}{\cos\left(\povn\right)}\ =\ L\cdot\tan\left(\povn\right).
}

Given that $a<\cot\left(\povn\right)$, there exists $a<L<b$ such that $L<\cot(\pi/n)$; then with sufficiently small $s,t$, all radii will be less than 1. That is, for small $s$ and $t$ there is an $r_0 > 0$ such that \eqn{ \{W_{1,0},W'_{1,0},W_{2,0},W'_{2,0}\}\text{ encircles any subset of }B(O_1,r_0)\cup B(O_2,r_0).} 

First choose $W_{1,0},W'_{1,0},W_{2,0},W'_{2,0}$ this way, only requiring that they are white. (As $a<L<b$, this is true for sufficiently small $s$ and~$t$.) 

Then, choose $W_{1,1},W'_{1,1}$ near $O_1$ and $W_{2,1},W'_{2,1}$ near $O_2$ such that \begin{itemize}
    \item their reflection points, the first two about $\ell_1$ and the last two about $\ell_2$, $B_{1,1},B'_{1,1},B_{2,1},B'_{2,1}$, are between $\ell_1,\ell_2$ and encircled by $W_{1,0},W'_{1,0},W_{2,0},W'_{2,0}$, \textit{and} 
    \item $\{W_{1,1},W'_{1,1},W_{2,1},W'_{2,1}\}$ also encircles points sufficiently close to $O_1,O_2$.
\end{itemize}

The first condition holds for the entire half-balls around $O_1,O_2$ (the half that is not between $\ell_1,\ell_2$); find $W_{1,1},W'_{1,1},W_{2,1},W'_{2,1}$ analogous to finding $W_{1,0},W'_{1,0},W_{2,0},W'_{2,0}$. This process can continue to assign $W_{1,i},W'_{1,i},W_{2,i},W'_{2,i},B_{1,i},B'_{1,i},B_{2,i},B'_{2,i}$ for every $i\in\N$. This collection of points satisfy the following: \begin{itemize}
    \item for all non-negative integers~$i$, $W_{1,i},W'_{1,i},W_{2,i},W'_{2,i}$ encircles $B_{1,i+1},B'_{1,i+1},B_{2,i+1},B'_{2,i+1}$, and
    \item for all non-negative integers~$i$, the rotation of $W_{1,i},W'_{1,i},B_{1,i},B'_{1,i}$ in counterclockwise direction by $2\pi/n$ about~$O$ is $B_{2,i},B'_{2,i},W_{2,i},W'_{2,i}$ in that order. That is, the points around $\ell_1$ have the same configuration as those around $\ell_2$, except that their colors are opposite.
\end{itemize}

The second property allows us to analogously define $W_{j,i},W'_{j,i},B_{j,i},B'_{j,i}$ around $\ell_j$ for $j=3,\dots,n$. The points around adjacent rays will have the same configuration but with opposite colors. By rotational symmetry, the first property becomes: For all odd $j\in[n]$,  and non-negative integers~$i$, the set of points $W_{j,i},W'_{j,i},W_{j+1,i},W'_{j+1,i}$ encircles $B_{j,i+1},B'_{j,i+1},B_{j+1,i+1},B'_{j+1,i+1}$.
Here $j+1$ is to $1$ when $j=n$.

Note that if we take $A$ to be $\overline{A}$ (i.e., flipping the colors) the argument above still holds, except that the parity of the ray indices will flip. Thus, for all non-negative integers~$i$, and for all $j\in[n]$, \eqn{
\begin{cases}
\{W_{j,i},W'_{j,i},W_{j+1,i},W'_{j+1,i}\}\text{ encircles }\{B_{j,i+1},B'_{j,i+1},B_{j+1,i+1},B'_{j+1,i+1}\} & (j\text{ is odd})\\
\{B_{j,i},B'_{j,i},B_{j+1,i},B'_{j+1,i}\}\text{ encircles }\{W_{j,i+1},W'_{j,i+1},W_{j+1,i+1},W'_{j+1,i+1}\} & (j\text{ is even}).
\end{cases}
}

For every non-negative integer~$i$, let $S_i\ :=\ \{W_{j,i},W'_{j,i},B_{j,i},B'_{j,i}:j\in[n]\}$. Since encirclement is closed under taking unions, we have that $S_i$ encircles $S_{i+1}$ for all non-negative integers~$i$.

If $A$ was drawable, for a non-negative integer~$i$ let $M_i:=\max_{x\in S_i}\stat(x)$. Thus $M_i$ is a positive integer. By Lemma~\ref{lem:maxsn}, \eqn{M_0\ >\ M_1\ >\ M_2\ >\ \cdots} which leads to contradiction. Hence $A$ is not drawable.
\end{proof}

\begin{proof}[Proof of Theorem \ref{thm:curv1}]
Return to {\snake}. By the construction (recall that $r:=1.001$), \begin{gather*}
    |BC|\ =\ 2r\tan30\degree\ =\ 2\sqrt3r,\\
    |AE|\ =\ \frac{|BC|}{2\cos 15\degree}-|BE|\ =\ \frac{4\sqrt3-\sqrt6-\sqrt2}{\sqrt6+\sqrt2}r\approx 0.793...\\
    |OE|\ =\ |AE|\cot15\degree\ =\ \frac{4\sqrt3-\sqrt6-\sqrt2}{\sqrt6-\sqrt2}r\approx 2.963...\\
    |OE'|\ =\ r\cot15\degree\ =\ \frac{\sqrt6+\sqrt2}{\sqrt6-\sqrt2}r\approx 3.735\dots
\end{gather*} 
We note that {\snake} is totally $12$-dissected at $(2.964, 3.735)$ with thickness $0.793$. In particular, $\cot(\pi/12)=2+\sqrt3\approx 3.732...$, so {\snake} is not drawable by Theorem~\ref{thm:totdisundraw}. By tweaking Definition \ref{defi:encircle} to deal with closed unit disks, the same argument shows that {\snake} is not closed-disk drawable. On the other hand, {\snake} is locally drawable (and locally closed-disk drawable) by Remark~\ref{rek:c2prob}.
\end{proof}
\begin{remark}
Theorem~\ref{thm:totdisundraw} can also be used to prove Theorem~\ref{thm:chessboard}: $[-c,0]^2\cup [0,c]^2$ is totally 4-dissected at $(0,c)$ with thickness~$c$, and $0<\cot(\pi/4)$.
\end{remark}

\section*{Acknowledgements}

The authors would like to thank Anton Bernshteyn, Clinton Conley and Junyao Peng 
for several helpful comments. FF was supported by NSF grant DMS 1855591 and a Sloan Research Fellowship. FP was supported by Carnegie Mellon's Summer Undergraduate Research Fellowship.

\bibliographystyle{plain}

\begin{thebibliography}{1}

\bibitem{balcerzak1999uncountable}
M.~Balcerzak and A.~Kharazishvili.
\newblock On uncountable unions and intersections of measurable sets.
\newblock {\em Georgian Mathematical Journal}, 6(3):201--212, 1999.

\bibitem{blaschke1916}
W.~Blaschke.
\newblock {\em {Kreis und Kugel}}.
\newblock Veit Leipzig, 1916.

\bibitem{doCarmo}
M.~Do~Carmo.
\newblock {\em Differential geometry of curves and surfaces}.
\newblock Courier Dover Publications, 2016.

\bibitem{dubins1957}
L.~Dubins.
\newblock On curves of minimal length with a constraint on average curvature,
  and with prescribed initial and terminal positions and tangents.
\newblock {\em American Journal of Mathematics}, 79(3):497--516, 1957.

\bibitem{keleti1999}
T.~Keleti.
\newblock {The Dynkin System Generated by the Large Balls of $\R^n$}.
\newblock {\em Real Analysis Exchange}, 24(2):859--866, 1999.

\bibitem{keleti2000}
T.~Keleti and D.~Preiss.
\newblock {The balls do not generate all Borel sets using complements and
  countable disjoint unions}.
\newblock {\em Mathematical Proceedings of the Cambridge Philosophical
  Society}, 128(3):539--547, 2000.

\bibitem{Srivastava_1998}
S.~M. Srivastava.
\newblock {\em A Course on Borel Sets}.
\newblock Springer Berlin Heidelberg, 1998.

\bibitem{Zeleny}
M.~Zelen\'{y}.
\newblock The {D}ynkin system generated by balls in {$\bold R^d$} contains all
  {B}orel sets.
\newblock {\em Proceedings of the American Mathematical Society},
  128(2):433--437, 2000.

\end{thebibliography}

\end{document}